\documentclass{amsart}

\usepackage{lineno}
\usepackage{amssymb}
\usepackage{amsmath}
\usepackage{amsthm}
\usepackage{mathrsfs}
\usepackage{bbm}
\usepackage{cite}
\usepackage{color}
\usepackage{enumitem}
\usepackage{pgfplots}

\usepackage[colorlinks,hypertexnames=false]{hyperref}
\usepackage[msc-links]{amsrefs}

\usepackage{tikz-cd}
\usepackage{mathtools}
\usepackage{extarrows}
\usepackage{setspace}

\hypersetup{
	colorlinks=true,
	linkcolor=black,
	filecolor=black,
	urlcolor=black,
	citecolor=black,
}

\newtheorem{thm}{Theorem}[section]
\newtheorem{cor}[thm]{Corollary}
\newtheorem{defn}[thm]{Definition}
\newtheorem{lem}[thm]{Lemma}
\newtheorem{prop}[thm]{Proposition}

\newtheorem{rmk}[thm]{Remark}

\numberwithin{equation}{section}

\begin{document}

\title[Weak slice regular functions on the $n$-dimensional quadratic cone]{Weak slice regular functions\\ on the $n$-dimensional quadratic cone of
octonions}
\author{Xinyuan Dou}
\email[Xinyuan Dou]{douxy@mail.ustc.edu.cn}
\address{Department of Mathematics, University of Science and Technology of China, Hefei 230026, China}
\date{\today}
\author{Guangbin Ren}
\email[Guangbin Ren]{rengb@ustc.edu.cn}
\address{Department of Mathematics, University of Science and Technology of China, Hefei 230026, China}
\date{\today}
\author{Irene Sabadini}
\email[Irene Sabadini]{irene.sabadini@polimi.it}
\address{Dipartimento di Matematica, Politecnico di Milano, Via Bonardi, 9, 20133 Milano, Italy}
\author{Ting Yang}
\email[Ting Yang]{tingy@aqnu.edu.cn}
\address{School of Mathematics and Computation Sciences, Anqing Normal University, Anqing 246133, China}
\date{\today}
\thanks{This work was supported by the NNSF of China (11771412). The third author is partially supported by the Project PRIN 2017 "Real and Complex Manifolds: Topology, Geometry
and holomorphic dynamics".}


\begin{abstract}
	In the literature on slice analysis in the hypercomplex setting, there are two main approaches to define slice regular functions in one variable:
one consists in requiring that the restriction to any complex plane is holomorphic (with the same complex structure of the complex plane), the second one makes use of stem and slice functions.
So far, in the setting of several hypercomplex variables, only the second
approach has been considered, i.e. the one based on stem functions. In this paper,
we use instead the first definition on the so-called $n$-dimensional quadratic cone of octonions.
 These two approaches yield the same class of slice regular functions  on
axially symmetric slice-domains, however, they are different on other types of domains. We call this new class of functions weak slice regular. We
show that there exist weak slice regular functions which are not slice regular in the second approach. Moreover, we study various properties of these functions, including a Taylor expansion.
\end{abstract}

\maketitle

\section{Introduction}
The theory of slice regular functions, also called slice monogenic or hyperholomorphic has considerably evolved, see \cite{Colombo2011001B,Colombobook2016,Gentili2013001B}, since these functions were first introduced by
Gentili and Struppa in \cite{Gentili2006001,Gentili2007001}. The function
theory had in fact older roots, see e.g. the work of Cullen \cite{Cullen1965001}, but it became a full function theory only recently. The study started for quaternionic functions and treating converging power series centered at the origin,  but it was soon realized that the definition in \cite{Gentili2006001,Gentili2007001} allowed to treat a general theory of these functions. The study was then generalized to the case of functions with values in a Clifford algebra see \cite{Colombo2009002,Colombo2010002},
and to octonions \cite{Gentili2010001}.
Roughly speaking, according to \cite{Gentili2006001}, a function is slice
regular if it is  holomorphic on each slice complex plane $\mathbb{C}_I$ for any imaginary unit $I$ of quaternions.
It turns out that many properties of holomorphic functions in complex analysis can be generalized to this framework.

 The function theory had a great impulse since it was soon realized that it allowed to define a functional calculus for quaternionic linear operators and for $m$-tuples of not necessarily commuting operators, see \cite{Colombo2011001B} and references therein, and it leads to the definition of the $S$-spectrum.  It has also applications  in Schur analysis \cite{Alpay2016001B,Alpaybook2020}, twistor transforms \cite{Gentili2014001},  and   operator theory \cite{Colombobook2018,Colombobook2019, Ghiloni2018001}.

A crucial tool in this type of analysis is the so-called representation formula, see \cite{Colombo2009Structure, Colombo2009003}, and then \cite{Colombo2009001} (for a version of the formula on semi-slices). This formula implies that all slice regular functions on open sets satisfying additional conditions.

The approach using slice functions and stems functions (notions which go back to  Fueter \cite{Fueter1934001})
 was considered by Ghiloni and Perotti in \cite{Ghiloni2011001} in the very general framework of real alternative algebras, which includes all the
previous cases, i.e. quaternions, Clifford algebras, octonions. The function theory so defined coincides with the previous one on some special open sets, as we shall see in Section 3.

We note that there is a third approach to slice hyperholomorphic functions using a global operator, see \cite{ColomboTAMS}, but we will not use this point of view in the present paper.

Soon after the introduction of the function theory of slice regular functions in one variable, started the study in several variables, and it was immediately clear that the situation was considerably more difficult. In the papers \cite{Colombo2012002,Ghiloni2012001}, the problem is addressed
following two different methods, but both based on stem functions. The first paper considers the quaternionic case, whereas the second one treats the Clifford variable case.
In this stream of ideas, Ren and Yang \cite{Ren2018001} define  slice regular functions on the $n$-dimensional quadratic cone of octonions $\mathbb{O}^n_s$, whereas the case of two quaternionic variables is treated in \cite{Dou2020002}.

	In \cite{Ghiloni2020001}, Ghiloni and Perotti define yet another kind of
slice regular functions on $(Q_A)^n$, where $A$ is a real alternative $*$-algebra; if $A=\mathbb{O}$, then $(Q_A)^n=\mathbb{O}^n$, but this class of functions is different from the one introduced by Ren and Yang \cite{Ren2018001}.

In this paper, we generalize some ideas in \cite{Dou2020001}, where we introduce the slice-topology in the space of quaternions, and we generalize
it to the $n$-dimensional quadratic cone $\mathbb{O}_s^n$ of octonions. This topology allows us to introduce a notion, that we call {\em weak slice regularity} on $\mathbb{O}_s^n$.
Functions weak slice regular turn out to be a generalization of  the   slice regular functions defined in \cite{Ren2018001}, which are called {\em
strong slice regular} in this paper.  Note that in one variable, weak slice regular functions coincide with the functions introduced by Gentili and Struppa, while strong slice regular functions coincide with those introduced by Ghiloni and Perotti.
Strong slice regular functions can be generalized to both $\mathbb{O}^n$ and $\mathbb{O}_s^n$,  while weak slice functions can be easily generalized to $\mathbb{O}_s^n$, but how to define them on $\mathbb{O}^n$ is an open question.

The  theory of weak slice regular functions is heavily based on the validity of the  representation formula, which is also the main tool for the extension of complex analysis to what is nowadays called  slice analysis.

As we pointed out, we will work on the  $n$-dimensional quadratic cone of
octonions $\mathbb{O}_s^n$ which is defined by
\begin{equation*}
\mathbb{O}_s^n=\bigcup_{I\in\mathbb{S}}\mathbb{C}_I^n,
\end{equation*}
where $\mathbb{C}_I^n=(\mathbb{C}_I)^n$, $I\in\mathbb O$ such that $I^2=-1$.
Note that since
\begin{equation*}
\mathbb{O}^n=\bigcup_{(I_1,\ldots, I_n)\in\mathbb{S}^n} \mathbb{C}_{I_1}\times\mathbb{C}_{I_2}\times\cdots\times\mathbb{C}_{I_n},
\end{equation*}
and the involved complex structures $I_1, \ldots, I_n$ are non commutative, it is a problem to provide  a suitable definition of holomorphy on the
slice $\mathbb{C}_{I_1}\times\mathbb{C}_{I_2}\times\cdots\times\mathbb{C}_{I_n}$.

Moreover,  strong slice regular functions in one variable are slice functions, and this fact makes it easier to generalize the notion to higher dimensions, see  \cite{Colombo2012002,Ghiloni2012001,Ren2018001,Dou2020002}.

Ren and Yang \cite{Ren2018001} defined the strong slice regular functions
on the $n$-dimensional quadratic cone of octonions $\mathbb{O}^n_s$. Since the slice-structure on $\mathbb{O}_s^n$ reflects the one of $\mathbb{O}$, many properties in complex analysis of several variables
 can be  generalized to  slice analysis on $\mathbb{O}_s^n$ using the property of being a slice function. As we already noted, unfortunately the slice analysis on $\mathbb{O}^n$ cannot be treated in this way.
 In contrast,  weak slice regular functions in one variable are not necessarily of slice type, and they may be defined on not necessarily axially symmetric slice-domains, see \cite{Dou2020002}.

	In \cite{Ghiloni2020001}, Ghiloni and Perotti define another kind of strong slice regular functions on the quadratic cone $(Q_A)^n$, where $A$ is
a real alternative $*$-algebra (if choose $A=\mathbb{O}$, then $(Q_A)^n=\mathbb{O}^n$). The restrictions of their functions on $\mathbb{O}_s^n$ are different from the functions introduced by Ren and Yang. Furthermore, the definition domain $\Omega$ of Ghiloni and Perotti's functions require stronger symmetric conditions.

In this paper we will consider weak slice regular functions on the $n$-dimensional quadratic cone of octonions $\mathbb{O}_s^n$. This class of functions allows to consider more general situations.
For example, we will show in the last section that the function $f(x)=\sqrt{x_1}$ can be extended to a weak slice regular on the slice-domain $\Omega$ defined by \eqref{eq-ob}. However, according to Proposition \ref{pr-tin}, $f$ cannot be extended to $\Omega$ as a strong slice regular function. We shall show, see Remark \ref{rm-sw}, that each strong slice regular extension is weak slice regular, but the converse does not hold, in general.
We prove that these two classes of functions coincide on axially symmetric domains in the slice-topology.

The structure of this paper is as follows.
In Section \ref{sc-pr}, we recall some basic concepts, also following \cite{Ren2018001}. In Section \ref{sc-sf}, we generalize slice functions on axially symmetric open sets to any subset in $\mathbb{O}_s^n$. We give the definition of slice function, including a version of the representation
formula; see Proposition \ref{pr-sf}. In Section \ref{sc-srf}, we introduce strong slice regular functions. In Section \ref{sc-ws}, we define the slice-topology on  $\mathbb{O}_s^n$ and the notion of weak slice regular functions on open sets in the slice-topology $\tau_s$. Moreover, we prove
a splitting lemma for these functions. We also show that strong slice regular functions are weak slice regular on axially symmetric open sets in $\tau_s$. In Section \ref{sc-ip}, we study in detail the properties of the
slice-topology on   $\mathbb{O}_s^n$ and we prove an identity principle for weak slice regular functions and for strong slice regular functions on
axially symmetric domains in $\tau_s$. In Section \ref{sc-rf}, we prove that the notion of weak slice regular and strong slice regular functions are equivalent on axially symmetric domains in $\tau_s$. We also give two representation formulas  for weak slice regular functions. Section 8 discusses the Taylor expansion in our framework. Finally, Section 9 contains an example of a weak slice regular function which is not (and cannot be extended to) a strong slice regular function.

\section{Preliminaries}\label{sc-pr}

In this section, we recall some well known material on octonions, see \cite{Schafer}, and some basic definitions  in octonionic slice analysis \cite{Ren2018001}.

The algebra of octonions $ \mathbb O $ is a real, alternative,  division algebra which is non-commutative and non-associative.
It can be generated  from the algebra of quaternions $ \mathbb H $ by the
Cayley-Dickson process.

As a real vector space, it is isomorphic to $ \mathbb R^8 $, and  can be equipped with a basis $ e_0=1,e_1,\dots,e_7 $.

The multiplication rules between elements in the basis $\{ e_0,e_1,\dots,e_7 \}$ are defined by the relations
$$ e_ie_j=-\delta_{ij}+\varepsilon_{ijk}e_k, \ \ \ \forall\  i, j, k \in \{ 1,2,\dots,7 \}. $$
Here $ \delta_{ij} $ is the Kronecker symbol and
\[ \varepsilon_{ijk}=
\begin{cases} (-1)^{\sigma(\pi)}  &\textrm{if}\ \   (i, j, k) \in \pi(\Sigma),
\\
0 & \textrm{otherwise}.
\end{cases} \]
where  $ \pi $ is a permutation, $\sigma(\pi)$ its sign, and
 $$ \Sigma=\{ (1, 2, 3), (1, 4, 5), (2, 4, 6), (3, 4, 7), (5, 3, 6), (6, 1,7), (7, 2, 5) \}. $$

The sphere $\mathbb{S}$ of imaginary units of octonions $\mathbb{O}$ is defined by
\begin{equation*}
\mathbb{S}:=\{I\in\mathbb{O}\ :\ I^2=-1\}.
\end{equation*}
An element $I\in\mathbb{S}$ belongs to the set of purely imaginary octonions namely octonions of the form $\sum_{\ell=1}^7 x_\ell e_\ell$, $x_\ell\in\mathbb{R}$.

For each $I\in\mathbb{S}$ we denote by $\mathbb{C}_I$ the complex plane whose imaginary unit is $I$ and setting $x=(x_1,...,x_n),y=(y_1,...,y_n)\in\mathbb{R}^n$, we write
\begin{equation*}
x+yI:=(x_1+y_1I,x_2+y_2I,...,x_n+y_n I), \qquad x_\ell+Iy_\ell\in\mathbb{C}_I, \, \ell=1,\ldots, n.
\end{equation*}

\begin{defn}
The $n$-dimensional quadratic cone in $\mathbb{O}^n$ is defined by
\begin{equation*}
	\mathbb{O}^n_s:=\bigcup_{I\in\mathbb{S}}\mathbb{C}_I^n
\end{equation*}
where
\begin{equation*}
\mathbb{C}_I^n:=\{x+yI:x,y\in\mathbb{R}^n\}
\end{equation*}
is called a {\em slice} (cartesian product of $n$ complex planes) of $\mathbb{O}^n_s$.
\end{defn}

The topology $\tau$ of $\mathbb{O}_s^n$ will be the subspace topology induced by the Euclidean topology of $\mathbb{O}^n$. For each $\Omega\subset\mathbb{O}^n_s$ and $I\in\mathbb{S}$, we set
\begin{equation*}
\Omega_I:=\Omega\cap\mathbb{C}_I^n
\end{equation*}
and we call it the $I$-slice (or a slice) of $\Omega$.

It is immediate to check that
\begin{equation*}
\mathbb{C}_I^n\cap\mathbb{C}_J^n=\mathbb{R}^n,\qquad\forall\ I,J\in\mathbb{S}\ \mbox{with}\ I\neq\pm J.
\end{equation*}

\begin{defn}
A set $\Omega\subset\mathbb{O}^n_s$ is called {\em axially symmetric} if for each $x,y\in\mathbb{R}^n$ and $I\in\mathbb{S}$ with $x+yI\in\Omega$,
	\begin{equation*}
	x+y\mathbb{S}\subset\Omega,
	\end{equation*}
	where $x+y\mathbb{S}:=\{x+yJ\in\mathbb{O}_s^n:x,y\in\mathbb{R}^{n}\ \mbox{and}\ J\in\mathbb{S}\}.$
\end{defn}

\section{Slice functions}\label{sc-sf}

In \cite{Ren2018001}, a class of slice regular functions is defined on the $n$-dimensional quadratic cone of octonions by slice functions and corresponding stem functions, following the slice analysis on real alternative $*$-algebras, see \cite{Ghiloni2011001}. This technique plays a fundamental role in slice analysis in one variable and was used recently in the octonionic case in a new way, connecting slice analysis with quaternionic
analysis, see \cite{Jin2020001}. In this section, we will generalize the slice functions defined in \cite{Ren2018001}.

We will use stem functions defined on the upper half complex plane $\mathbb{C}^+:=\{x+yi\in\mathbb{C}:y\ge 0\}$ identified with $\mathbb{R}\times\mathbb{R}_+$, where $\mathbb{R}_+:=[0,\infty)$. With this choice, we avoid to impose additional hypothesis on the domain and on the functions,
in order to have well posedness. Morever, we need  to extend the notion of upper half plane to $\mathbb{C}^n$. To this end, we introduce the next definition.

\begin{defn}
Given $x,y\in\mathbb{R}^n$, we say that $(x,y)\in\mathbb{R}^{2n}$ is {\em
positive} if $y_1>0$ or there is $m\in\{2,...,n\}$ such that  $y_m>0$ and
$y_\ell=0$ for all $1\leq\ell<m$. We say that the set
\begin{equation*}
\mathbb{R}^{2n}_+:=\{(x,y)\in\mathbb{R}^{2n}:(x,y)\ \mbox{is}\ \mbox{positive}\ \mbox{or}\ y=0\in\mathbb{R}^n\}
\end{equation*}
is the upper half-plane in $\mathbb{R}^{2n}$.
\end{defn}
Let us set
\begin{equation*}
\mathbb{R}^{2n}_-:=\{(x,y)\in\mathbb{R}^{2n}:(x,-y)\in\mathbb{R}^{2n}_+\}.
\end{equation*}
Then it is immediate that
\begin{equation*}
\mathbb{R}^{2n}_+\cup\mathbb{R}^{2n}_-=\mathbb{R}^{2n}\qquad\mbox{and}
\qquad\mathbb{R}^{2n}_+\cap\mathbb{R}^{2n}_-=\mathbb{R}^n\times\{0\}.
\end{equation*}
\\
In the sequel we will make use of the following useful notations: for any
$\Omega\subset\mathbb{O}_s^n$, define
\begin{equation*}
\Omega_{s_1}:=\{(x,y)\in\mathbb{R}^{2n}:\exists\ I\in\mathbb{S}\ s.t.\ x+yI\in\Omega\},
\end{equation*}
\begin{equation*}
\Omega_{s_2}:=\{(x,y)\in\mathbb{R}^{2n}:\exists\ J,K\in\mathbb{S}\ s.t.\ x+yJ,x+yK\in\Omega\},
\end{equation*}
\begin{equation*}
\Omega_{s_1}^+:=\Omega_{s_1}\cap\mathbb{R}^{2n}_+\qquad\mbox{and}\qquad\Omega_{s_2}^+:=\Omega_{s_2}\cap\mathbb{R}^{2n}_+.
\end{equation*}

We now define the slice functions in our setting:
\begin{defn}\label{df-sf}
	Let $\Omega\subset\mathbb{O}^n_s$. A function $f:\Omega\rightarrow\mathbb{O}$ is called {\em slice function} if there is $F:\Omega_{s_2}^+\rightarrow\mathbb{O}^{2\times 1}$ such that
	\begin{equation}\label{eq-s}
	f(x+yI)=(1,\, I)F(x,y)
	\end{equation}
	for any $(x,y)\in\Omega_{s_2}^+$ and any $I\in\mathbb{S}$ with $x+yI\in\Omega$.
\end{defn}

In the paper \cite{Dou2020001} the equality $(J-K)^{-1}K=-K(J-K)^{-1}$,
where $J,K$ are quaternions, is used in the proof of a general representation formula. We now replace the imaginary units $J,K$ by the left multiplication operators $L_J$, $L_K$ acting on $\mathbb{O}$ and we get the following equality:
\begin{equation}\label{eq-cjk}
(L_J-L_K)^{-1}L_J=-L_K(L_J-L_K)^{-1},\qquad\forall\ J,K\in\mathbb{S}\ \mbox{with}\ J\neq K,
\end{equation}
where
\begin{equation*}
L_I:\mathbb{O}\rightarrow\mathbb{O},\qquad o\mapsto Io,\qquad\forall\ o\in\mathbb{O}.
\end{equation*}
The validity of \eqref{eq-cjk} follows from
\begin{equation*}
L_I\circ L_I=-1,\qquad\forall\ I\in\mathbb{S},
\end{equation*}
where $1$ denotes here, and in the sequel, the identity map of $\mathbb{O}$.

It is easy to check that the following equality holds by \eqref{eq-cjk}
\begin{equation}\label{eq-mfm}
\left(\begin{matrix}
1&L_J\\1&L_K
\end{matrix}\right)^{-1}
=\left(\begin{matrix}
(L_J-L_K)^{-1}L_J&(L_K-L_J)^{-1}L_K\\(L_J-L_K)^{-1}&(L_K-L_J)^{-1}
\end{matrix}\right),
\end{equation}
for each $J,K\in\mathbb{S}$ with $J\neq K$.

In the next result we prove some equivalent statements characterizing slice functions following the ideas in quaternionic slice analysis, see \cite{Dou2020001}.
\begin{prop}\label{pr-sf}
	Let $\Omega\subset\mathbb{O}^n_s$ and $f:\Omega\rightarrow\mathbb{O}$. Then the following statements are equivalent:
	\begin{enumerate}[label=(\roman*)]
		
		\item $f$ is a slice function.
		
		\item (Stem function form) There is $F:\Omega_{s_1}\rightarrow\mathbb{O}^{2\times 1}$ such that
		\begin{equation}\label{eq-sff}
		f(x+yI)=(1,\, I)F(x,y)
		\end{equation}
		for each $(x,y)\in\Omega_{s_1}$ and $I\in\mathbb{S}$ with $x+yI\in\Omega$.
		
		\item (Matrix form) For each $(x,y)\in\mathbb{R}^{2n}$ and $I,J,K\in\mathbb{S}$ with $J\neq K$ and $x+yI,x+yJ,x+yK\in\Omega$,
		\begin{equation}\label{eq-mfs}
		f(x+yI)=(1,L_I)\left(\begin{matrix}
		1&L_J\\1&L_K
		\end{matrix}\right)^{-1}
		\left(\begin{matrix}
		f(x+yJ)\\f(x+yK)
		\end{matrix}\right).
		\end{equation}
		
		\item\label{it-sf4} (Linear form for $I$) For each $(x,y)\in\mathbb{R}^{2n}$ and $I,J,K\in\mathbb{S}$ with $J\neq K$ and $x+yI,x+yJ,x+yK\in\Omega$,
		\begin{equation}\label{eq-lfi}
		\begin{split}
		f(x+yI)=&(J-K)^{-1}[Jf(x+yJ)-Kf(x+yK)]\\&+I\{(J-K)^{-1}[f(x+yJ)-f(x+yK)]\}.
		\end{split}
		\end{equation}
		
		\item\label{it-sf5} (Linear from for $f$) For each $(x,y)\in\mathbb{R}^{2n}$ and $I,J,K\in\mathbb{S}$ with $J\neq K$ and $x+yI,x+yJ,x+yK\in\Omega$,
		\begin{equation}\label{eq-lff}
		f(x+yI)=(I-K)[(J-K)^{-1}f(x+yJ)]+(I-J)[(K-J)^{-1}f(x+yK)].
		\end{equation}
		
	\end{enumerate}
\end{prop}

\begin{proof}
	This proof exploits ideas in \cite{Dou2020001}. The fact that statements
\eqref{eq-mfs}, \eqref{eq-lfi} and \eqref{eq-lff} are equivalent follows by \eqref{eq-cjk} and \eqref{eq-mfm}. Then (iii), (iv) and (v) are also equivalent since they are basically a rewriting one of the others. We prove the equivalence of the remaining assertions.
	
	(i)$\Rightarrow$(ii). Suppose that $f$ is a slice function. Then, there is a function $G=(G_1,G_2)^T:\Omega_{s_2}^+\rightarrow\mathbb{O}^{2\times 1}$ such that
	\begin{equation*}
	f(x+yI)=(1,\, I)G(x,y)
	\end{equation*}
	for any $(x,y)\in\mathbb{R}^{2n}_+$ and $I\in\mathbb{S}$ with $x+yI\in\Omega$.

We define the function $F:\, \Omega_{s_1}\rightarrow\mathbb{O}^{2\times 1}$ by
\begin{equation*}
		F(x,y):=\left\{\begin{split}
			&(G_1,G_2)^T(x,y),&&\qquad (x,y)\in\Omega_{s_2}^+,
			\\&(G_1,-G_2)^T(x,-y),&&\qquad (x,y)\in\Omega_{s_2}^-\ \mbox{with}\ y\neq 0,
			\\&(f(x+yI_{x,y}),0)^T,&&\qquad (x,y)\in\Omega_{s_1}\backslash\Omega_{s_2},
		\end{split}\right.
	\end{equation*}

	Direct computations show that $F$ satisfies \eqref{eq-sff}, where $I_{x,y}$ is the unique imaginary unit $I\in\mathbb{S}$ such that $x+yI\in\Omega$. It follows that (ii) holds.
	
	(ii)$\Rightarrow$(iii). Suppose that $f$ satisfies (ii) and consider any
$(x,y)\in\mathbb{R}^{2n}$ and $I,J,K\in\mathbb{S}$ with $J\neq K$ and $x+yJ,x+yJ,x+yK\in\Omega$. Then \eqref{eq-sff} gives
	\begin{equation*}
	\left(\begin{matrix}
	f(x+yJ)\\f(x+yK)
	\end{matrix}\right)
	=\left(\begin{matrix}
	1&L_J\\1&L_K
	\end{matrix}\right)F(x,y),
	\end{equation*}
	from which we have
	\begin{equation}\label{eq-ms}
	F(x,y)=\left(\begin{matrix}
	1&L_J\\1&L_K
	\end{matrix}\right)^{-1}
	\left(\begin{matrix}
	f(x+yJ)\\f(x+yK)
	\end{matrix}\right).
	\end{equation}
	Then \eqref{eq-mfs} follows directly from \eqref{eq-sff} and \eqref{eq-ms}, so that (iii) holds.
	
	(iii)$\Rightarrow$(i). Suppose that $f$ satisfies (iii).  For any $(x,y)\in\Omega_{s_2}$, by the axiom of choice, there is $(J_{x,y},K_{x,y})\in\mathbb{S}$ with $J_{x,y}\neq K_{x,y}$ such that $x+yJ_{x,y},x+yK_{x,y}\in\Omega$. The function $F:\Omega_{s_2}\rightarrow\mathbb{O}^{2\times 1}$ defined by
	\begin{equation*}
	F(x,y):=\left(\begin{matrix}
	1&L_{J_{x,y}}\\1&L_{K_{x,y}}
	\end{matrix}\right)^{-1}
	\left(\begin{matrix}
	f(x+yJ_{x,y})\\f(x+yK_{x,y})
	\end{matrix}\right),\qquad\forall\ (x,y)\in\Omega_{s_2},
	\end{equation*}
	satisfies \eqref{eq-s}. Therefore $f$ is a slice function.
\end{proof}
\begin{defn}
Let $\Omega\subset\mathbb{O}^n_s$ and $f:\, \Omega\rightarrow\mathbb{O}$.
The function $F:\, \Omega_{s_1}\rightarrow\mathbb{O}^{2\times 1}$ satisfying  \eqref{eq-sff} is called {\em stem function}.
\end{defn}
\begin{rmk}
The linear forms for $I$ and for $f$ as in (iv), (v) appeared in \cite{Colombo2009001} and \cite{Ren2018001}.
\end{rmk}
In \cite{Ren2018001} the authors used a different notion of slice functions starting from functions defined on open subsets on $\mathbb C^n$ which
are intrinsic and then defining a slice function as a function induced by
an intrinsic one.
\begin{prop}
  \label{rm-se}
	The class of slice functions in Definition \ref{df-sf} and the class of slice functions in \cite{Ren2018001}  are equivalent on axially symmetric
domains.
\end{prop}
\begin{proof}
	Let $\Omega$ be an axially symmetric domain in $\mathbb{O}^n_s$. In the proof we will make use of the map
	\begin{equation*}
		\begin{split}
			\mathcal{P}_\mathbb{R}:\quad\mathbb{C}^n\quad &\xlongrightarrow[\hskip1cm]{}\quad \mathbb{R}^{2n},
			\\ x+yi\ &\shortmid\!\xlongrightarrow[\hskip1cm]{}\ (x,y).
		\end{split}
	\end{equation*}

	(i) Suppose $f:\Omega\rightarrow\mathbb{O}$ is a slice function according to the definition in \cite{Ren2018001}. Then there is a stem function $\mathcal{F}:\mathcal{P}_\mathbb{R}^{-1}(\Omega_{s_2})\rightarrow\mathbb{O}_\mathbb{C}$ such that
	\begin{equation}\label{eq-st}
	f(x+yI)=\mathcal{F}_1(x+yi)+I\mathcal{F}_2(x+yi)
	\end{equation}
	for each $x,y\in\mathbb{R}^n$ and $I\in\mathbb{S}$ with $x+yI\in\Omega$,
where $\mathcal{F}_1,\mathcal{F}_2:\mathcal{P}_\mathbb{R}^{-1}(\Omega_{s_2})\rightarrow\mathbb{O}$ with
	\begin{equation}\label{eq-p}
	\mathcal{F}=1\otimes \mathcal{F}_1+i\otimes \mathcal{F}_2.
	\end{equation}
	Then the function $F:\Omega_{s_2}^+\rightarrow\mathbb{O}^{2\times 1}$ defined by
	\begin{equation*}
	F(x,y):=(\mathcal{F}_1(x+yi),\mathcal{F}_2(x+yi))^T,\qquad\forall\ (x,y)\in\Omega_{s_2}^+,
	\end{equation*}
	satisfies \eqref{df-sf}, hence $f$ is slice in the sense of Definition \ref{df-sf}.
	
	(ii) Suppose that $f:\Omega\rightarrow\mathbb{O}$ is slice function according to Definition \ref{df-sf}. Then there is a function $F=(F_1,F_2)^T:\Omega_{s_2}^+\rightarrow\mathbb{O}^{2\times 1}$ satisfying \eqref{eq-s}, moreover $\mathcal{F}:\mathcal{P}_\mathbb{R}^{-1}(\Omega_{s_2})\rightarrow\mathbb{O}^{2\times 1}$ defined by
	\begin{equation}\label{eq-stc}
	\mathcal{F}(x+yi):=\left\{\begin{split}
	&1\otimes F_1(x,y)+i\otimes F_2(x,y),&&\qquad(x,y)\in\Omega_{s_1}^+,
	\\&1\otimes F_1(x,-y)-i\otimes F_2(x,-y),&&\qquad \mbox{otherwise},
	\end{split}\right.
	\end{equation}
	satisfies \eqref{eq-st}, as it can be verified by directly calculations.
By \eqref{eq-stc}, it is immediate to check that $\mathcal{F}$ is intrinsic, i.e.
	\begin{equation*}
	\mathcal{F}(x-yi)=1\otimes\mathcal{F}_1(x+yi)-i\otimes\mathcal{F}_2(x+yi),
	\end{equation*}
	 where $\mathcal{F}_1$ and $\mathcal{F}_2$ are defined by \eqref{eq-p}. Therefore, $f$ is a slice function in the sense of \cite{Ren2018001}.
\end{proof}

\begin{rmk}
	Note that the slice functions in \cite{Ren2018001} are defined on axially symmetric domains in $\mathbb{O}_s^n$, whereas Definition \ref{df-sf} allows to consider any subset of $\mathbb{O}_s^n$. Thus the class of slice
functions defined in Definition \ref{df-sf} is a generalization of the one considered in \cite{Ren2018001}.
\end{rmk}

By \cite[Proposition 2.5 and Corollary 2.6]{Ren2018001} and Proposition \ref{rm-se}, slice functions satisfies two representation formulas. These formulas are extended to our case, as consequences of Proposition \ref{pr-sf}.

\begin{cor}\label{co-grf}(General Representation Formula, linear form for
the functions)
	Let $\Omega\subset\mathbb{O}_s^n$ and $f:\Omega\rightarrow\mathbb{O}$ a slice function. Then
	\begin{equation*}
	f(x+yI)=(I-K)[(J-K)^{-1}f(x+yJ)]+(I-J)[(K-J)^{-1}f(x+yK)].
	\end{equation*}
	for each $(x,y)\in\mathbb{R}^{2n}$ and $I,J,K\in\mathbb{S}$ with $J\neq K$ and $x+yI,x+yJ,x+yK\in\Omega$.
\end{cor}

\begin{proof}
It follows from Proposition \ref{pr-sf} (i) and (v).
\end{proof}

\begin{cor}\label{co-rf}(Representation Formula, linear form for imaginary units)
	Let $\Omega\subset\mathbb{O}_s^n$ and $f:\Omega\rightarrow\mathbb{O}$ a slice function. Then
	\begin{equation}\label{eq-rf}
	f(x+yI)=\frac{1}{2}[f(x+yJ)+f(x-yJ)]-\frac{I}{2}\{J[f(x+yJ)-f(x-yJ)]\}.
	\end{equation}
	for each $(x,y)\in\mathbb{R}^{2n}$ and $I,J\in\mathbb{S}$ with $x+yI,x\pm yJ\in\Omega$.
\end{cor}

\begin{proof}
	It follows from Corollary \ref{co-grf}, by setting $J=K$.
\end{proof}

\section{Strong Slice regular functions}\label{sc-srf}
We recall that we will refer to the notion  given in \cite{Ren2018001} of
slice regular functions on axially symmetric open sets of $\mathbb{O}_s^n$ as strong slice regularity.
In this section we generalize strong slice regular functions from axially
symmetric open sets of $\mathbb{O}_s^n$ to any subsets of $\mathbb{O}_s^n$.

\begin{defn}
	Let $V\subset\mathbb{R}^{2n}$. We say that a  function $F:\, V\rightarrow\mathbb{O}^{2\times 1}$ is holomorphic if there is an open set $U$ in $\mathbb{R}^{2n}$ and a function $G:\, U\rightarrow\mathbb{O}^{2\times 1}$ with continuous partial derivatives such that $V\subset U$, $F|_V=G|_V$
and
	\begin{equation*}
	\frac{1}{2}(\frac{\partial}{\partial_{x_m}}+\sigma\frac{\partial}{\partial_{y_m}})G(x,y)=0,
	\end{equation*}
	for each $m\in\{1,2,...,n\}$ and $x=(x_1,...,x_n),y=(y_1,...,y_n)\in\mathbb{R}^n$ with $(x,y)\in U$, where
	\begin{equation*}
	\sigma:=\begin{pmatrix}
	0&-1\\1&0
	\end{pmatrix}
	\end{equation*}
	is a complex structure on $\mathbb{O}^{2\times 1}$ (resp. $\mathbb{R}^{2\times 1}$).
\end{defn}

\begin{defn}\label{df-sr}
	Let $\Omega\subset\mathbb{O}_s^n$. A function $f:\Omega\rightarrow\mathbb{O}$ is called strong slice regular if there is a holomorphic stem function $F:\Omega_{s_1}\rightarrow\mathbb{O}^{2\times 1}$ of $f$ i.e. there is a holomorphic $F:\Omega_{s_1}\rightarrow\mathbb{O}^{2\times 1}$ such that
	\[
	f(x+yI)=(1,\, I)F(x,y)
	\]	for any $(x,y)\in\Omega_{s_1}$ and any $I\in\mathbb{S}$ with $x+yI\in\Omega$.
\end{defn}

We define holomorphic functions on each slice $\mathbb{C}_I$ following the definition in complex analysis in the case of  vector-valued functions in several variables.

\begin{defn}
	Let $I\in\mathbb{S}$ and  $\Omega$ be an open set in $\mathbb{C}_I^n$. A
function $f:\Omega\rightarrow\mathbb{O}$ is said to be holomorphic, if $f$ has continuous partial derivatives and satisfies
	\begin{equation*}
	\frac{1}{2}(\frac{\partial}{\partial x_m}+I\frac{\partial}{\partial y_m})f(x+yI)=0
	\end{equation*}
	for all $m=1,2,...,n$, $x=(x_1,x_2,...,x_n),y=(y_1,y_2,...,y_n)\in\mathbb{R}^n$ with $x+yI\in\Omega$.
\end{defn}

\begin{prop}\label{pr-ws}
	Let $\Omega\subset\mathbb{O}_s^n$, $\Omega_I=\Omega\cap\mathbb{C}_I^n$
and $f:\Omega\rightarrow\mathbb{O}$ be a strong slice regular function. Then $f_I:=f|_{\Omega_I}$ is holomorphic for any $I\in\mathbb{S}$ with $\Omega_I\in\tau(\mathbb{C}_I^n)$.
\end{prop}

\begin{proof}
	Let $F:\Omega_{s_1}\rightarrow\mathbb{O}^{2\times 1}$ be a holomorphic stem function of $f$. Since $F$ has continuous partial derivatives, so is $f$ by \eqref{eq-sff}. From the equality
	\[
	I(1,I)=(1,I)\sigma.
	\]
	and \eqref{eq-sff} it follows that
	\begin{equation*}
	\frac{1}{2}(\frac{\partial}{\partial_{x_m}}+I\frac{\partial}{\partial_{y_m}})f(x+yI)=(1,I)[\frac{1}{2}(\frac{\partial}{\partial_{x_m}}+\sigma\frac{\partial}{\partial_{y_m}})]F(x+yI)=0
	\end{equation*}
	for all $m\in\{1,2,...,n\}$, $(x,y)\in\Omega_{s_1}$ and $I\in\mathbb{S}$
with $x+yI\in\Omega$. It is clear that $f_I$ is holomorphic.
\end{proof}

\section{Weak slice regular functions}\label{sc-ws}

In this section, we generalize the slice-topology on $\mathbb{H}$ (see \cite{Dou2020001}) to $\mathbb{O}_s^n$. Moreover, we will define another class of slice regular functions that we call weak slice regular, following
the terminology in \cite{Dou2020001}. This class of functions is defined by imposing the condition of being holomorphic on each slice $\mathbb{C}_I$, $I\in\mathbb{C}$, instead of requiring that the stem functions are holomorphic. As it is well known, these two approaches give, in general, different classes of functions, unless one imposes additional conditions. The two classes are same on some axially symmetric sets. We now set
$$\tau_s(\mathbb{O}_s^n):=\{U\subset\mathbb{O}_s^n:U_I\in\tau(\mathbb{C}_I^n),\ \forall\ I\in\mathbb{S}\}.$$
It is easy to check that $\tau_s$ is a topology on $\mathbb{O}_s^n$, thus
leading to the following:
\begin{defn}
	We call $\tau_s$ the slice-topology on $\mathbb{O}_s^n$, where $\tau(\mathbb{C}_I^n)$ is the Euclidean topology on $\mathbb{C}^n_I$.
\end{defn}

 In the sequel, open sets, connectedness and paths in the slice-topology are called slice-open sets, slice-connectedness and slice-paths, respectively, and similarly for other notions.

\begin{defn}
	Let $\Omega$ be a slice-open set in $\mathbb{O}^n_s$. A function $f:\Omega\rightarrow\mathbb{O}$ is called weak slice regular, if for any $I\in\mathbb{S}$, the restriction $f_I:=f|_{\Omega_I}$ is holomorphic.
\end{defn}

\begin{rmk}\label{rm-sw}
	By Proposition \ref{pr-ws}, any strong slice regular functions defined on a slice-open set is weak slice regular.
\end{rmk}

\begin{defn}
The set	$\{I,J,K\}$ is called a s-basis of $\mathbb{O}$, if $I,J,K\in\mathbb{S}$ and
	\begin{equation*}
		\{1,I,J,IJ,K,JK,IK,IJK\}
	\end{equation*}
	is a basis of $\mathbb{O}$ as a real algebra.
\end{defn}

By decomposing the values of a function using a fixed s-basis, we can prove the following:

\begin{lem}\label{lm-sp}
	(Splitting Lemma)
	Let $f$ be a function defined on a slice-open subset $\Omega$ of $\mathbb{O}_s^n$. Then $f$ is weak slice regular if and only if for any s-basis $\{I,J,K\}$, there are four complex-valued holomorphic functions $F_1,F_2,F_3,F_4: \Omega_I\rightarrow\mathbb{C}_I$ such that
	\begin{equation*}
	f_I=F_1+F_2J+F_3K+F_4(JK).
	\end{equation*}
\end{lem}

\begin{proof}
	(i) Suppose that $f$ is weak slice regular. Then for each s-basis $\{I,J,K\}$, there are four complex-valued functions $F_1,F_2,F_3,F_4: \Omega_I\rightarrow\mathbb{C}_I$ such that
	\begin{equation*}
	f_I=F_1+F_2J+F_3K+F_4(JK).
	\end{equation*}
	If $f$ is weak slice regular, then
		\begin{equation*}
		\begin{split}
			&\overline{\partial_I} F_1+\left[\overline{\partial_I} F_2\right]J+\left[\overline{\partial_I} F_3\right]K+\left[\overline{\partial_I} F_4\right](JK)
			\\=&\overline{\partial_I} F_1+\overline{\partial_I} [F_2J]+\overline{\partial_I} [F_3K]+\overline{\partial_I} [F_4(JK)]
			=\overline{\partial_I} f_I=0,
		\end{split}
	\end{equation*}
	where
	\begin{equation*}
	\overline{\partial_I}:=\begin{pmatrix}
	\frac{1}{2}(\frac{\partial}{\partial x_1}+I\frac{\partial}{\partial y_1})
	\\\vdots
	\\\frac{1}{2}(\frac{\partial}{\partial x_n}+I\frac{\partial}{\partial y_n}).
	\end{pmatrix}
	\end{equation*}
	Since $\mathbb{O}=\mathbb{C}_{I}\oplus\mathbb{C}_{I}J\oplus\mathbb{C}_{I}K\oplus\mathbb{C}_{I}(JK)$, we have
	\begin{equation*}
	\overline{\partial_I} F_1=\overline{\partial_I} F_2=\overline{\partial_I} F_3=\overline{\partial_I} F_4=0.
	\end{equation*}
	It follows that $F_1,F_2,F_3,F_4$ are holomorphic.
	
	(ii) Suppose that for any s-basis $\{I,J,K\}$,  there are four complex-valued holomorphic functions $F_1,F_2,F_3,F_4: \Omega_I\rightarrow\mathbb{C}_I$ such that $f_I:=f|_{\Omega_I}$ can be written as
	\begin{equation*}
	f_I=F_1+F_2J+F_3K+F_4(JK).
	\end{equation*}
	Then
	\begin{equation*}
	\overline{\partial_I} f_I=\overline{\partial_I} F_1+\overline{\partial_I} F_2J+\overline{\partial_I} F_3K+\overline{\partial_I} F_4(JK)=0.
	\end{equation*}
Hence $f_I$ is holomorphic for all $I\in\mathbb{S}$. It implies that $f$ is weak slice regular.
\end{proof}

\section{Identity principle}\label{sc-ip}
In this section we prove an identity principle for both the weak and strong slice regular functions. Moreover, we generalize some properties of the slice-topology
of $\mathbb{H}$ (see \cite{Dou2020001}) to $\mathbb{O}_s^n$.

\begin{defn}
	A subset $\Omega$ of $\mathbb{O}^n_s$ is called real-connected, if $\Omega_{\mathbb{R}}:=\Omega\cap\mathbb{R}^n$ is connected in $\mathbb{R}^n$.
\end{defn}

We note that the subspace topology on $\mathbb{R}^n$ induced by $\tau_s$ coincides with the Euclidean topology on $\mathbb{R}^n$, and so the notions of slice-connectedness and of connectedness coincide on $\mathbb{R}^n$.

\begin{prop}\label{pr-rs}
	Let $\Omega$ be a slice-open set in $\mathbb{O}^n_s$. Then for any $q\in\Omega$, there is a real-connected slice-domain $V\subset\Omega$ containing $q$.
\end{prop}

\begin{proof}
	Let us consider $q\in\Omega$ and a set $A$ defined as follows: if $q\in\mathbb{R}^n$ then $A$ is the connected component of $\Omega_{\mathbb{R}}$
containing $q$ in $\mathbb{R}^n$, otherwise $A=\varnothing$. Then the slice-connected component $(\Omega\backslash\Omega_{\mathbb{R}})\cup A$ containing $q$ is a real-connected slice-domain in $\Omega$ containing $q$.
\end{proof}

\begin{defn}
	A path $\gamma$ in $\mathbb{O}^n_s$ is called on a slice, if $\gamma\subset\mathbb{O}^n_s\cap\mathbb{C}_I^n$ for some $I\in\mathbb{S}$.
\end{defn}

\begin{prop}
Any path on a slice is a slice-path.
\end{prop}

\begin{proof}
	It follows directly from the fact that $\tau_s(\mathbb{C}_I^n)=\tau(\mathbb{C}_I^n)$ for any $I\in\mathbb{S}$.
\end{proof}

\begin{prop}\label{pr-psa}
	Let $\Omega$ be a real-connected slice-domain in $\mathbb{O}^n_s$. Then
	
	\begin{enumerate}[label=(\roman*)]
		
		\item If $\Omega_\mathbb{R}=\varnothing$, then $\Omega\subset\mathbb{C}_{I}^n$ for some $I\in\mathbb{S}$.
		
		\item If $\Omega_\mathbb{R}\neq\varnothing$, then for each $q\in\Omega$
and $x\in\Omega_\mathbb{R}$, there is a slice-path $\gamma$ on a slice from $q$ to $x$.
		
		\item For all $I\in\mathbb{S}$, $\Omega_I$ is a domain in $\mathbb{C}_I^n$.
		
		\item For all $p,q\in\Omega$, there exist two slice paths $\gamma_1,\gamma_2$, each of which on a slice in $\Omega$, such that $\gamma_1\gamma_2$ is a path from $p$ to $q$.
	\end{enumerate}
\end{prop}

\begin{proof}
	(i)	Let us assume that $\Omega_\mathbb{R}=\varnothing$. Then for any $I\in\mathbb{S}$, $\mathbb{C}_I^n\backslash\mathbb{R}^n$ is slice-open in $\mathbb{O}^n_s$ and
		\begin{equation*}
			\Omega\subset\bigsqcup_{I\in\mathbb{S}^+}(\mathbb{C}_{I}^n\backslash\mathbb{R}^n),
		\end{equation*}
		where $\mathbb{S}^+$ is a subset of $\mathbb{S}$ such that $\mathbb{S}=\mathbb{S}^+\bigsqcup(-\mathbb{S}^+)$.
		By the slice-connectedness of $\Omega$, $\Omega\subset\mathbb{C}_{I}^n\backslash\mathbb{R}^n$ for some $I\in\mathbb{S}$.
	
	(ii) For any $q\in\Omega$  there is $I\in\mathbb{S}$ such that $q\in\mathbb{C}_I^n$. Set $V$ the connected component of $\Omega_I$ containing $q$. If $V_\mathbb{R}=\varnothing$, then $V$ and $\Omega\backslash V$ are slice-open, and since $\Omega$ is slice-connected, it follows that $V=\Omega$ and $\Omega_\mathbb{R}=\varnothing$,
	which is a contradiction. Therefore, $V_\mathbb{R}\neq\varnothing$. Let now $x\in\Omega_{\mathbb R}$. Then for any $x_0\in V_\mathbb{R}$, there is a path $\alpha$ in $V$ from $q$ to $x_0$, and a path $\beta$ in $\Omega_\mathbb{R}$ from $x_0$ to $x$. It follows that the path composition $\alpha\beta$ is a slice-path on a slice from $q$ to $x$.
	
	(iii), (iv) follow directly from (i) and (ii).
\end{proof}

\begin{prop}\label{pr-lp}
	The topological space $(\mathbb{O}_s^n,\tau_s)$ is connected, locally path-connected and path-connected.
\end{prop}

\begin{proof}
	Propositions \ref{pr-rs} and \ref{pr-psa} (iv) imply that $(\mathbb{O}_s^n,\tau_s)$ is locally path-connected. Since
	\begin{equation*}
	\mathbb{O}^n_s\cap\mathbb{C}_I^n=\mathbb{C}_I^n\supset\mathbb{R}^n,\qquad\forall\ I\in\mathbb{S},
	\end{equation*}
	it is clear that $(\mathbb{O}_s^n,\tau_s)$ is path-connected and connected.
\end{proof}

\begin{lem}\label{th-dps}
	Let $\Omega$ be a real-connected slice-domain in $\mathbb{O}_s^n$ and $f,g$ be two weak slice regular functions on $\Omega$. If $f$ and $g$ coincide on a nonempty open subset of $\Omega_I$ (or $\Omega_{\mathbb{R}}$) for some $I\in\mathbb{S}$, then $f=g$ on $\Omega$.
\end{lem}

\begin{proof}
We begin by observing that Proposition \ref{pr-psa} (iii) and the fact that  $\Omega_I\neq\varnothing$ imply that $\Omega_I$ is a nonempty domain in $\mathbb{C}_I^n$.
	
	(i) Suppose $f$ and $g$ coincide on an open subset $V$ of $\Omega_{\mathbb{R}}$. Then for any $x\in V$, $f$ and $g$ have the same Taylor series at $x \in \Omega_{\mathbb{R}}$. According to Splitting Lemma \ref{lm-sp}, $f$ and $g$ have the same Taylor series at $x$ in $\Omega_I$, for any fixed $I\in\mathbb{S}$ since this is true for each complex valued components
of $f$ and $g$. Let $r\in\mathbb{R}_+$ such that $B_I(x,r)\subset\Omega_I$, where $B_I(x,r)$ is the ball with center $x$ and radius $r$ in $\mathbb{C}_I^n$. Then
	\begin{equation*}
		\begin{split}
		f^{(I,\alpha)}(x)=&\left[\frac{1}{2}\left(\frac{\partial}{\partial x_1}-I\frac{\partial}{\partial y_1}\right)\right]^{\alpha_1}\cdots\left[\frac{1}{2}\left(\frac{\partial}{\partial x_n}-I\frac{\partial}{\partial y_n}\right)\right]^{\alpha_n}f(x)
		\\=&\left(\frac{\partial}{\partial x_1}\right)^{\alpha_1}\cdots\left(\frac{\partial}{\partial x_n}\right)^{\alpha_n}f(x)=\left(\frac{\partial}{\partial x_1}\right)^{\alpha_1}\cdots\left(\frac{\partial}{\partial x_n}\right)^{\alpha_n}g(x)=g^{(I,\alpha)}(x),
		\end{split}
\end{equation*}
where $f^{(I,\alpha)}:=(f|_{\Omega_I})^{(I,\alpha)}$ (the meaning of the symbols $(f|_{\Omega_I})^{(I,\alpha)}$ is clear, but see Definition \ref{df-lis}). Hence for any $z\in B_I(x,r)$,
\begin{equation*}
	f(z)=\sum_{\alpha\in\mathbb{N}^n}\frac{1}{\alpha !}(z-x)^\alpha f^{(I,\alpha)}(x)=\sum_{\alpha\in\mathbb{N}^n}\frac{1}{\alpha !}(z-x)^\alpha g^{(I,\alpha)}(x)=g(z).
\end{equation*}
Therefore $f=g$ on an open set $B_I(x,r)$ in $\Omega_I$.

Therefore $f$ and $g$ coincide in $\Omega_I$. Consequently, $f=g$ on $\Omega=\bigcup_{J\in\mathbb{S}}\Omega_J$.

	(ii) Suppose there is $I\in\mathbb{S}$ such that $f$ and $g$ coincide on
an open subset of $\Omega_I$. Then $f$ and $g$ coincide on $\Omega_I$. If
$\Omega_{\mathbb{R}}=\varnothing$, then $\Omega=\Omega_I$ by Proposition \ref{pr-psa} (i), hence this lemma holds in this case. Otherwise $\Omega_{\mathbb{R}}\neq\varnothing$, then the lemma holds by (i).
\end{proof}

\begin{thm}\label{th-idh}
	(Identity Principle) Let $\Omega$ be a slice-domain  in $\mathbb{O}_s^n$
and $f,g$ be weak or strong slice regular functions on $\Omega$. If $f$ and $g$ coincide on a nonempty open subset $V$ of $\Omega_I$ (or $\Omega_{\mathbb{R}}$) for some $I\in\mathbb{S}$ then $f=g$ in $\Omega$.
\end{thm}

\begin{proof}
	(i) Suppose $f,g$ are weak slice regular. Set
	\begin{equation*}
		A:=\{x\in\Omega:\exists\ U\in\tau_s(\Omega),\ \mbox{s.t.}\ x\in U\ \mbox{and}\ f=g\ \mbox{on}\ U\}.
	\end{equation*}
	
	Obviously, $A$ is slice-open in $\Omega$ by definition.
	
	Choose $q\in V$. According to Proposition \ref{pr-rs}, there is a real-connected slice-domain $U$ containing $q$ with $V\subset\Omega$. By Lemma \ref{th-dps}, $f=g$ on $V$. Hence $q\in A$ and $A$ is nonempty.
	
	Choose $q\in\Omega\backslash A$. Set $U$ be a real-connected slice-domain in $\mathbb{O}_s^n$ containing $q$ with $U\subset\Omega$. It is clear that $A\cap U$ is slice-open in $\Omega$. If $A\cap U\neq\varnothing$, then $f=g$ on $A\cap U$. By Lemma \ref{th-dps}, $f=g$ on $V$ and $q\in A$, which is a contradiction. Therefore $A\cap V=\varnothing$ and $q$ is
a slice-interior of $\Omega\backslash A$. It implies that $\Omega\backslash A$ is slice-open, and since $\Omega$ is slice-connected, $A=\Omega$.
Hence $f=g$ on $\Omega$.
	
	(ii) Suppose $f,g$ are strong slice regular. By Remark \ref{rm-sw}, $f,g$ are weak slice regular and so the assertion holds by (i).
\end{proof}

\section{Representation formulas for weak slice regular functions}\label{sc-rf}

In this section, we prove that weak slice regular functions defined on axially symmetric slice-domains are slice functions. Hence they satisfy the
representation formulas in Corollaries \ref{co-grf} and \ref{co-rf}. Moreover, we show that a function defined on an axially symmetric slice-domain is weak slice regular if and only if it is strong slice regular.

The following result is very simple but crucial in the sequel:
\begin{prop}\label{pr-vn}
	Let $\Omega$ be a nonempty axially symmetric slice-domain. Then $\Omega_\mathbb{R}\neq\varnothing$.
\end{prop}
\begin{proof}
It follows from the definition.
\end{proof}
In the sequel, we shall always assume that $\Omega\subset \mathbb{O}_s^n$
is nonempty.
\begin{thm}\label{tm-ws}
	Weak slice regular functions defined on axially symmetric slice-domains are slice functions.
\end{thm}

\begin{proof}
	Let $\Omega$ be an axially symmetric slice-domain in $\mathbb{O}_s^n$ and let $f:\Omega\rightarrow\mathbb{O}$ be weak slice regular. Then $\Omega_{\mathbb{R}}\neq\varnothing$, by Proposition \ref{pr-vn}.

	Fix $I\in\mathbb{S}$. We can define the functions $g[J]:\Omega_J\rightarrow\mathbb{O}$, $J\in\mathbb{S}$, by
	\begin{equation}\label{eq-gjy}
		g[J](x+yJ)=\frac{1}{2}[f(x+yI)+f(x-yI)]-\frac{J}{2}\{I[f(x+yI)-f(x-yI)]\},
	\end{equation}
	for any $x+yJ\in\Omega_J$. Note that for any $x\in\Omega_\mathbb{R}$ we have
	\begin{equation*}
		g[J](x)=\frac{1}{2}[f(x)+f(x)]-\frac{J}{2}\{I[f(x)-f(x)]\}=f(x),
	\end{equation*}
	and for any $x+yJ\in\Omega$
	\begin{equation*}
		\begin{split}
			&g[-J](x+yJ)=g[-J](x+(-y)(-J))
			\\=&\frac{1}{2}[f(x-yI)+f(x+yI)]-\frac{-J}{2}\{I[f(x-yI)-f(x+yI)]\}
			\\=&\frac{1}{2}[f(x+yI)+f(x-yI)]-\frac{J}{2}\{I[f(x+yI)-f(x-yI)]\}=g[J](x+yJ).
		\end{split}
	\end{equation*}
	It is clear that there is a function $h:\Omega\rightarrow\mathbb{O}$ such that
	\begin{equation}\label{eq-hxy}
		h(x+yJ)=g[J](x+yJ),\qquad\forall\ x+yJ\in\Omega.
	\end{equation}
Moreover,	\eqref{eq-gjy}, \eqref{eq-hxy} show that the function $F:\Omega_{s_1}\rightarrow\mathbb{O}^{2\times 1}$ defined by
	\begin{equation}\label{eq-st1}
		F(u,v)=\begin{pmatrix}
			\frac{1}{2}[f(u+vI)+f(u-vI)]
			\\
			\quad \\ -\frac{1}{2}I[f(u+vI)-f(u-vI)]
		\end{pmatrix},\qquad\forall\ (u,v)\in\Omega_{s_1},
	\end{equation}
	is a stem function of $h$ so that $h$ is slice.
	
	By directly calculation, for each $\ell\in\{1,2,...,n\}$ and $x+yJ\in\Omega$.
	\begin{equation*}
		\begin{split}
			&\frac{1}{2}(\frac{\partial}{\partial x_\ell}+J\frac{\partial}{\partial y_\ell})g[J](x+yJ)
			\\=&\frac{1}{2}(\frac{\partial}{\partial x_\ell}+J\frac{\partial}{\partial y_\ell})\bigg(\frac{1}{2}[f(x+yI)+f(x-yI)]
			-\frac{J}{2}\{I[f(x+yI)-f(x-yI)]\}\bigg)
			\\=&\frac{1-JI}{4}\left((\frac{\partial}{\partial x_\ell}+I\frac{\partial}{\partial y_\ell})f(x+yI)\right)
			+\frac{1+JI}{4}\left((\frac{\partial}{\partial x_\ell}+(-I)\frac{\partial}{\partial y_\ell})f(x+y(-I))\right)
			\\=&0.
		\end{split}
	\end{equation*}
	It implies that $g[J]$ is holomorphic, and so is $h_J$. Hence $h$ is weak slice regular. Since $h=g[I]=f$ on $\Omega_\mathbb{R}$, it follows from the Identity Principle \ref{th-idh} that $f=g$ on $\Omega$. Then $f$ is slice, since $g$ is slice.

\end{proof}

\begin{cor}(General Representation Formula)
	Let $\Omega$ be an axially symmetric slice-domain in $\mathbb{O}_s^n$ and $f:\Omega\rightarrow\mathbb{O}$ be a weak slice regular function. Then
	\begin{equation*}
	f(x+yI)=(I-K)[(J-K)^{-1}f(x+yJ)]+(I-J)[(K-J)^{-1}f(x+yK)].
	\end{equation*}
	for each $(x,y)\in\mathbb{R}^{2n}$ and $I,J,K\in\mathbb{S}$ with $J\neq K$ and $x+yI,x+yJ,x+yK\in\Omega$.
\end{cor}

\begin{proof}
It is an immediate consequence of Corollary \ref{co-grf} and Theorem \ref{tm-ws}.
\end{proof}

\begin{cor}(Representation Formula)\label{cor-repr}
	Let $\Omega$ be an axially symmetric slice-domain in $\mathbb{O}_s^n$ and $f:\Omega\rightarrow\mathbb{O}$ a weak slice regular function. Then
	\begin{equation*}
	f(x+yI)=\frac{1}{2}[f(x+yJ)+f(x-yJ)]-\frac{I}{2}\{J[f(x+yJ)-f(x-yJ)]\}.
	\end{equation*}
	for each $(x,y)\in\mathbb{R}^{2n}$ and $I,J\in\mathbb{S}$ with $x+yI,x\pm yJ\in\Omega$.
\end{cor}

\begin{proof}
It is an immediate consequence of Corollary \ref{co-rf} and Theorem \ref{tm-ws} and also a particular case of Corollary \ref{cor-repr} when $K=-J$.
\end{proof}

\begin{lem}\label{lm-ws}
	Weak slice regular functions defined on axially symmetric slice-domains are strong slice regular.
\end{lem}

\begin{proof}
	Let $\Omega$ be an axially symmetric slice-domain in $\mathbb{O}_s^n$ and let $f:\Omega\rightarrow\mathbb{O}$ be a weak slice-regular with stem function $F:\Omega_{s_1}\rightarrow\mathbb{O}^{2\times 1}$ defined by \eqref{eq-st1}. It is clear by \eqref{eq-st1}, that $F$ has continuous partial derivatives.
	
	Let us fix $I\in\mathbb{S}$. Using \eqref{eq-mfm} we can rewrite \eqref{eq-st1} as
	\begin{equation}\label{eq-ffs}
	F(x,y)=\begin{pmatrix}
	1&I\\1&-I
	\end{pmatrix}^{-1}
	\begin{pmatrix}
	f(x+yI)\\f(x-yI)
	\end{pmatrix},\qquad\forall\ (x,y)\in\Omega_{s_1}.
	\end{equation}
	
Since
	\begin{equation}\label{eq-ig}
	\begin{pmatrix}
	1&I\\1&-I
	\end{pmatrix}\sigma
	=\begin{pmatrix}
	I&0\\0&-I
	\end{pmatrix}
	\begin{pmatrix}
	1&I\\1&-I
	\end{pmatrix},
	\end{equation}
	by \eqref{eq-ffs} and \eqref{eq-ig}, it follows that
	\begin{equation}\label{eq-fh}
	\frac{1}{2}(\frac{\partial}{\partial x_\ell}+\sigma\frac{\partial}{\partial y_\ell})F(x,y)=
	\begin{pmatrix}
	1&I\\1&-I
	\end{pmatrix}^{-1}
	\begin{pmatrix}
		\frac{1}{2}\left(\frac{\partial}{\partial x_\ell}+I\frac{\partial}{\partial y_\ell}\right)f(x+Iy)
\\ \quad		\\ \frac{1}{2}(\frac{\partial}{\partial x_\ell}+(-I)\frac{\partial}{\partial y_\ell})f(x+(-I)y)
	\end{pmatrix}
	\end{equation}
	for $\ell\in\{1,2,...,n\}$ and $x,y\in\Omega_{s_1}$. Since $f$ is holomorphic on $\Omega_I$ (resp. $\Omega_{-I}$), $F$ is holomorphic by \eqref{eq-fh} and we conclude that $f$ is strong slice regular.
\end{proof}

\begin{thm}\label{th-mt}
The class of weak slice regular and strong slice regular functions coincide on axially symmetric slice-domains.
\end{thm}

\begin{proof}
	The result follows by Proposition \ref{pr-ws} and Lemma \ref{lm-ws}.
\end{proof}

\section{Taylor series}

An important feature of holomorphic functions in several complex variables is the fact that they can be expanded as Taylor series. A function theory generalizing the one of holomorphic functions is expected to have this
property and in fact, in this section, we prove that  weak slice regular functions admit a Taylor series expansion.

\subsection{Slice derivatives}

To start with, we generalize the notion of slice derivative to the case of functions in several octonionic variables.

\begin{defn}\label{df-lis}
	Let $I\in\mathbb{S}$, $\Omega$ be an open set in $\mathbb{C}_I^n$, $\ell\in\{1,...,n\}$ and $f:\Omega\rightarrow\mathbb{O}$ be real differentiable. The $I,\ell$-derivative $\partial_{I,\ell}f:\Omega\rightarrow\mathbb{R}^{2n}$ of $f$ is defined by
	\begin{equation*}
		\partial_{I,\ell}f(x+yI):=\frac{1}{2}\left(\frac{\partial}{\partial x_\ell}-I\frac{\partial}{\partial y_\ell}\right)f_I(x+yI).
	\end{equation*}
	
Let $\alpha=(\alpha_1,...,\alpha_n)\in\mathbb{N}^n$. We set
	\begin{equation*}
		f^{(I,\alpha)}:=(\partial_{I,1})^{\alpha_1}\cdots(\partial_{I,n})^{\alpha_n} f,
	\end{equation*}
when the right-hand side is defined.
\end{defn}

\begin{prop}
	Let $I\in\mathbb{S}$, $\Omega$ be an open set in $\mathbb{C}_I^n$, and $f:\Omega\rightarrow\mathbb{O}$ be real differentiable. Then
	\begin{equation*}
		\partial_{I,\ell}f=\partial_{-I,\ell}f,\qquad \forall\ \ell\in\{1,...,n\}.
	\end{equation*}

	Moreover, if $f$ is holomorphic, then
	\begin{equation}\label{eq-pie}
		\partial_{I,\ell}f=\frac{\partial}{\partial x_\ell} f,\qquad \forall\
\ell\in\{1,...,n\}.
	\end{equation}
\end{prop}

\begin{proof}
	(i) By direct calculation, for each $\ell\in\{1,...,n\}$ and $x+yI\in\Omega$,
	\begin{equation*}
		\begin{split}
			\partial_{I,\ell}f(x+yI)=&\frac{1}{2}\left(\frac{\partial}{\partial x_\ell}-I\frac{\partial}{\partial y_\ell}\right)f(x+yI)
			\\=&\frac{1}{2}\left(\frac{\partial}{\partial x_\ell}-(-I)\frac{\partial}{\partial (-y_\ell)}\right)f(x+(-y)(-I))
			\\=&\partial_{-I,\ell}f(x+yI).
		\end{split}
	\end{equation*}

	(ii) Suppose that $f$ is holomorphic. Then for each $\ell\in\{1,...,n\}$
and $x+yI\in\Omega$,
	\begin{equation*}
		\frac{1}{2}\left(\frac{\partial}{\partial x_\ell}+I\frac{\partial}{\partial y_\ell}\right)f(x+yI)=0.
	\end{equation*}
	It implies that
	\begin{equation*}
		\partial_{I,\ell}f(x+yI)
		=\frac{\partial}{\partial x_\ell}f(x+yI)
		-\frac{1}{2}\left(\frac{\partial}{\partial x_\ell}+I\frac{\partial}{\partial y_\ell}\right)f(x+yI)
		=\frac{\partial}{\partial x_\ell}f(x+yI).
	\end{equation*}
\end{proof}

The set of purely imaginary octonions is denoted by $Im(\mathbb{O})$ and we observe that
\begin{equation*}
	Im(\mathbb{O})=\{tI\in\mathbb{O}:t\in\mathbb{R},\ I\in\mathbb{S}\}.
\end{equation*}
It is easy to check that $Im(\mathbb{O})$ is a $7$-dimensional real vector space and that
\begin{equation*}
	\mathbb{O}^n=\mathbb{R}^{n}\oplus Im(\mathbb{O})^n.
\end{equation*}

\begin{defn} (Slice derivatives)
	Let $\Omega$ be a slice-open set in $\mathbb{O}_s^n$,  and $f:\Omega\rightarrow\mathbb{O}$ be weak slice regular. The $\ell$-slice derivative $\partial_\ell f:\Omega\rightarrow \mathbb{O}$ of $f$, where $\ell\in\{1,...,d\}$, is defined by
	\begin{equation*}
		\partial_\ell f (x+w)=\frac{\partial}{\partial x_\ell} f(x+w),
	\end{equation*}
	where $x=(x_1,...,x_n)\in\mathbb{R}^{n}$ and $w\in Im(\mathbb{O})^n$.
	
	Let $\alpha=(\alpha_1,...,\alpha_n)\in\mathbb{N}^n$, we set
	\begin{equation*}
		f^{(\alpha)}:=(\partial_{1})^{\alpha_1}\cdots(\partial_{n})^{\alpha_n} f.
	\end{equation*}
\end{defn}

\begin{prop}
	Let $\Omega$ be a slice-open set in $\mathbb{O}_s^n$ and $f:\Omega\rightarrow\mathbb{O}$ be weak slice regular. Then for each $\alpha\in\mathbb{N}^n$, $f^{(\alpha)}$ is weak slice regular and
	\begin{equation}\label{eq-fa}
		\left(f^{(\alpha)}\right)_I=\left(f_I\right)^{(I,\alpha)},\qquad\forall\ I\in\mathbb{S}.
	\end{equation}	
\end{prop}

\begin{proof}
	The proof is by induction on the length $|\alpha|=\sum_{i=1}^n \alpha_i$ of $\alpha$. For $|\alpha|=0$, it is clear that $f^{(0)}$ is weak slice regular and \eqref{eq-fa} holds. Suppose that $f^{(\beta)}$ is weak
slice regular, $|\beta|=m$ and \eqref{eq-fa} holds when $\alpha=\beta$. We shall prove the statement for multi-indices of lenght $m+1$, i.e. we shall prove that for all $\ell\in\{1,...,n\}$, $f^{(\beta+\theta_\ell)}$ is weak slice regular and \eqref{eq-fa} holds when $\alpha=\beta+\theta_\ell$, where
	\begin{equation*}
		\theta_\ell=\left(0_{1\times(\ell-1)},1,0_{1\times(n-\ell)}\right).
	\end{equation*}
	
	Let $\ell\in\{1,...,n\}$ and $I\in\mathbb{S}$. Since $f^{(\beta)}$ is weak slice regular, $\left(f^{(\beta)}\right)_I$ is holomorphic. By induction hypothesis and \eqref{eq-pie} we deduce that
	\begin{equation*}
		\left(f^{(\beta+\theta_\ell)}\right)_I
		=\left(\frac{\partial}{\partial x_\ell}f^{(\beta)}\right)
		=\left(\frac{\partial}{\partial x_\ell}(f_I)^{(I,\beta)}\right)
		=(f_I)^{(I,\beta+\theta_\ell)}
	\end{equation*}
	is holomorphic. Since the choice of $I$ is arbitrary, it follows that $f^{(\beta+\theta_\ell)}$ is weak slice regular and \eqref{eq-fa} holds when $\alpha=\beta+\theta_\ell$.
	
	Thus we proved that for each $\alpha\in\mathbb{N}^n$, $f^{(\alpha)}$ is weak slice regular and \eqref{eq-fa} holds.
\end{proof}

\subsection{Further properties of octonions}

Before to state our main results, we prove some useful properties of octonions.

Below, the symbol $\langle\cdot,\cdot\rangle$ denotes the scalar product in $\mathbb O$.
\begin{prop}
	Let $r,s,p\in\mathbb{O}$. Then
	\begin{equation}\label{eq-bsl}
		\langle rp,s\rangle=\langle r,s\overline{p}\rangle.
	\end{equation}
\end{prop}

\begin{proof}
	If $p=0$, it is clear that \eqref{eq-bsl} holds. Otherwise,
	\begin{equation*}
		\begin{split}
			\langle rp,s\rangle
			=&\frac{1}{4}\left(|rp+s|^2-|rp-s|^2\right)
			\\=&\frac{1}{4}\left(\left|r+s|p|^{-2}\overline{p}\right|^2|p|^2-\left|r-s|p|^{-2}\overline{p}\right|^2|p|^2\right)
			\\=&\frac{1}{4}\left(\left|r+|p|^{-2}s\overline{p}\right|^2-\left|r-|p|^{-2}s\overline{p}\right|^2\right)|p|^2
			\\=&\langle r,|p|^{-2}s\overline{p}\rangle|p|^2
			=\langle r,s\overline{p}\rangle.
		\end{split}
	\end{equation*}
\end{proof}

Let $q\in\mathbb{O}$. There are unique $r\in\mathbb{R}$ and $s\in Im(\mathbb{O})$ such that
\begin{equation*}
	q=r+s,
\end{equation*}
and this fact justifies the notation
\begin{equation*}
	Re(q):=r\qquad\mbox{and}\qquad Im(q):=s.
\end{equation*}

\begin{prop}\label{eq-rso}
	Let $r,s\in\mathbb{O}$. Suppose that $r\overline{s}\in\mathbb{C}_I$ for some $I\in\mathbb{C}$. Then for any $J\in\mathbb{S}$,
	\begin{equation}\label{eq-mk}
		\min_{K=\pm I}(|r+Ks|)\le|r+Js|\le\max_{K=\pm I}(|r+Ks|).
	\end{equation}
\end{prop}

\begin{proof}
	If $r=0$ or $s=0$, then
	\begin{equation*}
		|r+J s|=|r+K s|
	\end{equation*}
	for each $J,K\in\mathbb{S}$. Then \eqref{eq-mk} holds in this case.
	
	Otherwise, let $J\in\mathbb{S}$. It is clear by \eqref{eq-bsl} that
	\begin{equation}\label{eq-tj}
		t_J:=Re\left((Js)r^{-1}\right)=\left\langle Js,\overline{r^{-1}}\right\rangle=\left\langle J,|r|^{-2}r\overline{s}\right\rangle.
	\end{equation}
	Then
	\begin{equation}\label{eq-rjs}
		\begin{split}
			|r+Js|=&\left|1+(Js)r^{-1}\right||r|=\left[(1+t_J)^2+(C^2-t_J^2)\right]|r|
			\\=&[1+C^2+2t_J]|r|,
		\end{split}
	\end{equation}
	where $C:=|(Js)r^{-1}|=|s|/|r|$ is a constant and $C^2-t_J^2=\left|Im\left((Js)r^{-1}\right)\right|^2$.
	
	Suppose that $r\overline{s}\in\mathbb{C}_I$ for some $I\in\mathbb{C}$. It is easy to check by \eqref{eq-tj} that
	\begin{equation*}
		\min_{K=\pm I}(t_K)\le t_J\le\max_{K=\pm I}t_K,
	\end{equation*}
	and this implies, from \eqref{eq-rjs}, that \eqref{eq-mk} holds.
\end{proof}

\subsection{Taylor series}

In this subsection, we shall prove that  weak slice regular functions can
be expanded in Taylor series on slice-polydiscs.

Let $I\in\mathbb{S}$ and $r=(r_1,...,r_n)\in\mathbb{R}_+^n=(0,+\infty]^n$. For each $z=(z_1,...,z_n)\in\mathbb{C}_I^n$, denote the polydisc with center $z$ and radius $r$ by
\begin{equation*}
	\mathbb{P}_I(z,r):=\left\{w=(w_1,...,w_n)\in\mathbb{C}_I^n:|z_\ell-w_\ell|\le r_\ell,\ \ell=1,...,n\right\},
\end{equation*}
and we set
\begin{equation*}
	\widetilde{\mathbb{P}}_I(z,r):=\{x+yJ\in\mathbb{O}_s^n:J\in\mathbb{S},\ x\pm yI\in\mathbb{P}_I(z,r)\}\bigcup\mathbb{P}_I(z,r).
\end{equation*}
It is easy to check that if $z\in\mathbb{C}_I\cap\mathbb{C}_J$, for some $I,J\in\mathbb{S}$, then
\begin{equation*}
	\widetilde{\mathbb{P}}_I(z,r)=\widetilde{\mathbb{P}}_J(z,r).
\end{equation*}
Hence, we can write $\widetilde{\mathbb{P}}(z,r)$ short for $\widetilde{\mathbb{P}}_I(z,r)$, without ambiguity. We call $\widetilde{\mathbb{P}}(z,r)$ the slice-polydisc with center $z$ and radius $r$.

\begin{prop}\label{pr-ts}
	Let $I\in\mathbb{S}$, $\Omega$ be an open set in $\mathbb{C}_I^n$ and $f:\Omega\rightarrow\mathbb{R}^{2n}$ be holomorphic. Then for each $z_0\in\Omega$ and $r\in\mathbb{R}_+^n$ with $\mathbb{P}_I(z_0,r)\subset\Omega$, we have
	\begin{equation*}
		f(z)=\sum_{\alpha\in\mathbb{N}^d}\frac{1}{\alpha!}(z-z_0)^\alpha f^{(I,\alpha)}(z_0),\qquad\forall\ z\in\mathbb{P}_I(z_0,r).
	\end{equation*}
\end{prop}

\begin{proof}
	This proposition follows directly from the Splitting Lemma \ref{lm-sp} and the Taylor expansion for holomorphic functions in several complex analysis.
\end{proof}

Let $\alpha=(\alpha_1,...,\alpha_n)\in\mathbb{N}^n$ and $\beta=(\beta_1,...,\beta_n)\in\mathbb{N}^n$. We say that $\alpha<\beta$ if
\begin{equation*}
	\alpha_\ell<\beta_\ell,\qquad\forall\ \ell=1,...,n.
\end{equation*}

Let $p\in\mathbb{O}_s^n$. We define the map
	\begin{equation*}
	\begin{split}
		(id-p)^{*\alpha}:\quad\mathbb{O}_s^n\ &\xlongrightarrow[\hskip1cm]{}\ End(\mathbb{O}),
		\\ q\ &\shortmid\!\xlongrightarrow[\hskip1cm]{}\ \sum_{0\le\beta\le\alpha}\left[
		\begin{pmatrix}
			\alpha\\\beta
		\end{pmatrix}L_q^\beta L_p^{\alpha-\beta}\right],
	\end{split}
\end{equation*}
where $id(q)=q$,
\begin{equation*}
	\begin{pmatrix}
		\alpha\\\beta
	\end{pmatrix}:=\begin{pmatrix}
		\alpha_1\\\beta_1
	\end{pmatrix}\cdots\begin{pmatrix}
		\alpha_n\\\beta_n
	\end{pmatrix}
\end{equation*}
is the binomial coefficient and
\begin{equation*}
	L_q^\beta:=\prod_{\ell=1}^n\left(L_{q_\ell}\right)^{\beta_\ell}=\left(L_{q_1}\right)^{\beta_1}\cdots \left(L_{q_n}\right)^{\beta_n}.
\end{equation*}
Below,  we write $(q-q_0)^{*\alpha}$ instead of $(id-p)^{*\alpha}(q)$ and
$(q-p)^{*\alpha}a$ instead of $[(q-p)^{*\alpha}](a)$.
\begin{prop}\label{pr-lpn}
	Let $p\in\mathbb{O}_s^n$ and $a\in\mathbb{O}$. The function defined by
	\begin{equation*}
		\begin{split}
			f:\quad\mathbb{O}_s^n\ &\xlongrightarrow[\hskip1cm]{}\ \mathbb{O},
			\\ q\ &\shortmid\!\xlongrightarrow[\hskip1cm]{}\ (q-p)^{*\alpha}a,
		\end{split}
	\end{equation*}
	is a weak slice regular function.
	
	Moreover, if $p\in\mathbb{C}_I^n$ and $q=x+yJ$ for some $I,J\in\mathbb{S}$ and $x,y\in\mathbb{R}^n$, then
	\begin{equation}\label{eq-mlq}
		\min_{r=x\pm yI} \left|(r-p)^{*\alpha}a\right|
		\le \left|(q-p)^{*\alpha}a\right|\le\max_{r=x\pm yI}\left|(r-p)^{*\alpha}a\right|.
	\end{equation}
\end{prop}

\begin{proof}
	(i) Let $K\in\mathbb{S}$ and $z=x+yK\in\mathbb{C}_K^n$. For any $\beta\in\mathbb{N}^n$, $b\in\mathbb{O}$ and $\ell=\{1,...,n\}$, we have
	\begin{equation*}
		\begin{split}
			&\frac{1}{2}\left(\frac{\partial}{\partial x_\ell}+K\frac{\partial}{\partial y_\ell}\right)(L_z^\beta b)
			\\=&\left[\frac{1}{2}\left(\frac{\partial}{\partial x_\ell}+L_K\frac{\partial}{\partial y_\ell}\right)(x_\ell+y_\ell L_K)^{\beta_\ell}\right]\left(\prod_{\jmath\neq\ell}(L_{z_\jmath})^{\beta_\jmath}b\right)=0.
		\end{split}
	\end{equation*}
We deduce that for any $\ell\in\{1,...,n\}$,
	\begin{equation*}
		\begin{split}
			&\frac{1}{2}\left(\frac{\partial}{\partial x_\ell}+K\frac{\partial}{\partial y_\ell}\right)f_K(z)
			\\=&\frac{1}{2}\left(\frac{\partial}{\partial x_\ell}+K\frac{\partial}{\partial y_\ell}\right)(z-p)^{*\alpha}a
			\\=&\sum_{0\le\beta\le\alpha}\left[\frac{1}{2}\left(\frac{\partial}{\partial x_\ell}+K\frac{\partial}{\partial y_\ell}\right)
			\begin{pmatrix}
				\alpha\\\beta
			\end{pmatrix}L_z^\beta \left(L_p^{\alpha-\beta}a\right)\right]=0.
		\end{split}
	\end{equation*}
	So $f_K$ is holomorphic  and since the choice of $z$ and $K$ is arbitrary, $f$ is weak slice regular.
	
	(ii) Since $f$ is weak slice regular and $\mathbb{O}_s^n$ is axially symmetric, it follows from Lemma \ref{lm-ws} that $f$ is strong slice regular and so $f$ is also slice. By Proposition \ref{pr-sf} \ref{it-sf4},
	\begin{equation*}
		(q-p)^{*\alpha}a=f(q)=f(x+yJ)=s+Jr,
	\end{equation*}
	where
	\begin{equation*}
		\begin{cases}
			r=\frac{1}{2}\left[f(x+yI)+f(x-yI)\right],
			\\s=\frac{1}{2}[If(x+yI)-If(x-yI)].
		\end{cases}
	\end{equation*}
	It is easy to check that
	\begin{equation*}
		r=w_1 a\qquad\mbox{and}\qquad s=w_2 a,
	\end{equation*}
	for some $w_1,w_2\in\mathbb{C}_I$. It is clear that
	\begin{equation*}
		r\overline{s}=|a|^2 w_1\overline{w_2}\in\mathbb{C}_I.
	\end{equation*}
	Hence \eqref{eq-mlq} holds directly by Proposition \ref{eq-rso}.
\end{proof}

\begin{thm} (Taylor series)
	Let $\Omega$ be a slice-open set in $\mathbb{O}_s^n$ and $f:\Omega\rightarrow\mathbb{O}$ be weak slice regular. Then for any $q_0\in\Omega$ and $r\in\mathbb{R}_+^n$ with $\widetilde{\mathbb{P}}(q_0,r)\subset\Omega$, we
have
	\begin{equation}\label{eq-fqsa}
		f(q)=\sum_{\alpha\in\mathbb{N}^d}\frac{1}{\alpha!}(q-q_0)^{*\alpha} f^{(\alpha)}(q_0),\qquad q\in\widetilde{\mathbb{P}}(q_0,r).
	\end{equation}
\end{thm}

\begin{proof}
	Let $q=x+yJ\in\widetilde{\mathbb{P}}(q_0,r)$. Suppose that $q_0\in\mathbb{C}_I^n$ for some $I\in\mathbb{S}$. If $J=\pm I$, then \eqref{eq-fqsa} holds by Proposition \ref{pr-ts}. Otherwise, by definition, $x\pm yI\in\widetilde{\mathbb{P}}(q_0,r)$. According to \eqref{eq-mlq}, for each $\alpha\in\mathbb{N}^n$,
	\begin{equation}\label{eq-faq0}
		\begin{split}
			\left|(q-q_0)^{*\alpha} f^{(\alpha)}(q_0)\right|\le\left|(q_I-q_0)^{*\alpha} f^{(\alpha)}(q_0)\right|+\left|(q_{-I}-q_0)^{*\alpha} f^{(\alpha)}(q_0)\right|
			\\=\left|(q_I-q_0)^{\alpha} f^{(I,\alpha)}(q_0)\right|+\left|(q_{-I}-q_0)^{\alpha} f^{(I,\alpha)}(q_0)\right|,
		\end{split}
	\end{equation}
	where $q_I=x+yI$ and $q_{-I}=x-yI$. Since
	\begin{equation*}
		\sum_{\alpha\in\mathbb{N}^n}\left|\frac{1}{\alpha!}(q_I-q_0)^{\alpha} f^{(I,\alpha)}(q_0)\right|\qquad\mbox{and}\qquad\sum_{\alpha\in\mathbb{N}^n}\left|\frac{1}{\alpha!}(q_{-I}-q_0)^{\alpha} f^{(I,\alpha)}(q_0)\right|
	\end{equation*}
	are convergent at $q_I$ and $q_{-I}$, respectively, it follows from \eqref{eq-faq0} that
	\begin{equation*}
		\sum_{\alpha\in\mathbb{N}^n}\left|\frac{1}{\alpha!}(q-q_0)^{*\alpha} f^{(\alpha)}(q_0)\right|\qquad\mbox{and}\qquad \sum_{\alpha\in\mathbb{N}^n}\frac{1}{\alpha!}(q-q_0)^{*\alpha}
f^{(\alpha)}(q_0)
	\end{equation*}
	are also convergent at $q$.
	
	By Proposition \ref{pr-lpn}, the function $g:\widetilde{\mathbb{P}}(q_0,r)\rightarrow\mathbb{O}$ defined by
	\begin{equation*}
		g(q):=\sum_{\alpha\in\mathbb{N}^n}\frac{1}{\alpha!}(q-q_0)^{*\alpha} f^{(\alpha)}(q_0)
	\end{equation*}
	is a weak slice regular function. Note that
	\begin{equation*}
		f=g,\qquad\mbox{on}\qquad\mathbb{P}_I(q_0,r),
	\end{equation*}
	and $f(q)$ and $g(q)$ are weak slice regular. It follows from the Identity Principle \ref{th-idh} that
	$f=g$,
	i.e. \eqref{eq-fqsa} holds.
\end{proof}

\section{Example}\label{sc-ex}

In this section we give an example of a weak slice regular function which
is not strong slice regular. In fact that there are weak slice regular functions which are not strong slice regular, not even on non-axially symmetric slice-domains, see \cite[Section 8]{Dou2020001}. To construct examples is more complicated, but basically they can be obtained by moving the singularities of the function $\widetilde{f}$ constructed below away from
$\mathbb{R}^n$.

Define two paths $\alpha,\beta:[0,1]\rightarrow\mathbb{C}^n$ by
\begin{equation*}
\alpha(t)=(e^{\frac{i\pi t}{2}},0,...,0)\qquad\mbox{and}\qquad\beta(t)=(e^{\frac{i\pi (t+1)}{2}},0,...,0),\qquad\forall\ t\in[0,1].
\end{equation*}
Let $\theta_1:=(1,0,...,0)$. It is easy to check that $\alpha$ is a path from form $\theta_1$ to $i\theta_1$ and $\beta$ is a path form $-\theta_1$ to $-i\theta_1$. Set
	\begin{equation*}
		U:=\{z\in\mathbb{C}^n:\mbox{dist}(z,\alpha)<\frac{1}{2}\},\qquad V:=\{z\in\mathbb{C}^n:\mbox{dist}(z,\beta)<\frac{1}{2}\},
	\end{equation*}
	\begin{equation*}
		U':=\{x+yi\in\mathbb{C}^n:x-yi\in U\}\quad\mbox{and}\quad V':=\{x+yi\in\mathbb{C}^n:x-yi\in V\}.
	\end{equation*}
Choose $I,J\in\mathbb{S}$ with $J\neq\pm I$. Then
\begin{equation}\label{eq-ob}
\Omega:=B(\theta_1,\frac{1}{2})\cup B(-\theta_1,\frac{1}{2})\cup\mathcal{P}_J( U'\cup V'\cup V)\cup\mathcal{P}_I(U),
\end{equation}
is a slice-domain, where $B(\theta_1,\frac{1}{2}):=\{o\in\mathbb{O}_s^n:|o-\theta_1|\le 1/2\}$. One may check that the function
\begin{equation}\label{eq-fbl}
	\begin{split}
		f:B\left(\theta_1,\frac{1}{2}\right)\bigcap\mathbb{R}^n &\xlongrightarrow[\hskip1cm]{}\quad \mathbb{R},
		\\ x=(x_1,...,x_n)\ &\shortmid\!\xlongrightarrow[\hskip1cm]{}\ \sqrt{x_1}
	\end{split}
\end{equation}
has a weak slice regular extension $\widetilde{f}$ on $\Omega$ with
\begin{equation}\label{eq-slf}
\widetilde{f}(\theta_1 J)=\frac{-1-J}{\sqrt{2}},\qquad\widetilde{f}(-\theta_1J)=\frac{1-J}{\sqrt{2}}\qquad
\mbox{and}\qquad\widetilde{f}(-\theta_1I)=\frac{1+I}{\sqrt{2}}.
\end{equation}
Simple computations using \eqref{eq-slf} show that $\widetilde{f}$ does not satisfy \eqref{eq-rf}. Hence $\widetilde{f}$ is not a slice function and it is not strong slice regular, although $\widetilde{f}$ is weak slice
regular.

\begin{prop}\label{pr-tin}
	There is no strong slice regular extension of
		$f$, where $f$ is defined in \eqref{eq-fbl}.
\end{prop}

\begin{proof}
	We prove this proposition by contradiction. Suppose that there is a strong slice regular extension $f':\Omega\rightarrow\mathbb{O}$ of $f$. By Remark \ref{rm-sw}, $f'$ is weak slice regular. According to Theorem \ref{th-idh}, $f'=\widetilde{f}$. However, $f'=\widetilde{f}$ is not strong
slice regular, which is a contradiction.
\end{proof}

%\bibliographystyle{amsplane}
%\bibliographystyle{bibstyle}
%\bibliography{mybibfile}

\end{document}